\documentclass[amssymb, 11pt]{amsart}
\usepackage{latexsym}
\usepackage{amsmath}
\usepackage[margin=.5in]{geometry}
\usepackage[round,authoryear]{natbib}

\usepackage{etoolbox} 
\usepackage{mathtools} 

\makeatletter
\def\cite{\citet}

\numberwithin{equation}{section}
\allowdisplaybreaks

\def\@noindentfalse{\global\let\if@noindent\iffalse}
\def\@noindenttrue {\global\let\if@noindent\iftrue}
\def\@aftertheorem{%
  \@noindenttrue
  \everypar{%
    \if@noindent%
      \@noindentfalse\clubpenalty\@M\setbox\z@\lastbox%
    \else%
      \clubpenalty \@clubpenalty\everypar{}%
    \fi}}

\theoremstyle{plain}
\newtheorem{theorem}{Theorem}[section]
\AfterEndEnvironment{theorem}{\@aftertheorem}

\AfterEndEnvironment{definition}{\@aftertheorem}

\AfterEndEnvironment{lemma}{\@aftertheorem}
\newtheorem{corollary}[theorem]{Corollary}
\AfterEndEnvironment{corollary}{\@aftertheorem}

\theoremstyle{definition}
\newtheorem{remark}[theorem]{Remark}
\AfterEndEnvironment{remark}{\@aftertheorem}

\AfterEndEnvironment{assumption}{\@aftertheorem}

\AfterEndEnvironment{example}{\@aftertheorem}

\AfterEndEnvironment{proof}{\@aftertheorem}

\newlength{\standardunitlength}
\def\be#1{\begin{equation*}#1\end{equation*}}
\def\ben#1{\begin{equation}#1\end{equation}}
\def\bes#1{\begin{equation*}\begin{split}#1\end{split}\end{equation*}}

\def\bs#1{\begin{split}#1\end{split}}
\def\ba#1{\begin{align*}#1\end{align*}}

\usepackage{delimset}
\def\given{\mskip 0.5mu plus 0.25mu\vert\mskip 0.5mu plus 0.15mu}
\newcounter{bracketlevel}%
\def\@bracketfactory#1#2#3#4#5#6{%
\expandafter\def\csname#1\endcsname##1{%
\global\advance\c@bracketlevel 1\relax%
\global\expandafter\let\csname @middummy\alph{bracketlevel}\endcsname\given%
\global\def\given{\mskip#5\csname#4\endcsname\vert\mskip#6}\csname#4l\endcsname#2##1\csname#4r\endcsname#3%
\global\expandafter\let\expandafter\given\csname @middummy\alph{bracketlevel}\endcsname%
\global\advance\c@bracketlevel -1\relax%
}%
}
\def\bracketfactory#1#2#3{%
\@bracketfactory{#1}{#2}{#3}{relax}{0.5mu plus 0.25mu}{0.5mu plus 0.15mu}
\@bracketfactory{b#1}{#2}{#3}{big}{1mu plus 0.25mu minus 0.25mu}{0.6mu plus 0.15mu minus 0.15mu}
\@bracketfactory{bb#1}{#2}{#3}{Big}{2.4mu plus 0.8mu minus 0.8mu}{1.8mu plus 0.6mu minus 0.6mu}
\@bracketfactory{bbb#1}{#2}{#3}{bigg}{3.2mu plus 1mu minus 1mu}{2.4mu plus 0.75mu minus 0.75mu}
\@bracketfactory{bbbb#1}{#2}{#3}{Bigg}{4mu plus 1mu minus 1mu}{3mu plus 0.75mu minus 0.75mu}
}
\bracketfactory{clc}{\lbrace}{\rbrace}
\bracketfactory{clr}{(}{)}
\bracketfactory{cls}{[}{]}
\bracketfactory{abs}{\lvert}{\rvert}
\bracketfactory{norm}{\Vert}{\Vert}
\bracketfactory{floor}{\lfloor}{\rfloor}
\bracketfactory{ceil}{\lceil}{\rceil}
\bracketfactory{angle}{\langle}{\rangle}

\let\original@left\left
\let\original@right\right
\renewcommand{\left}{\mathopen{}\mathclose\bgroup\original@left}
\renewcommand{\right}{\aftergroup\egroup\original@right}

\newcounter{ctr}\loop\stepcounter{ctr}\edef\X{\@Alph\c@ctr}%
	\expandafter\edef\csname s\X\endcsname{\noexpand\mathscr{\X}}
	\expandafter\edef\csname c\X\endcsname{\noexpand\mathcal{\X}}
	\expandafter\edef\csname b\X\endcsname{\noexpand\boldsymbol{\X}}
	\expandafter\edef\csname I\X\endcsname{\noexpand\mathbb{\X}}
\ifnum\thectr<26\repeat

0\loop\stepcounter{ctr}\edef\X{\@alph\c@ctr}%
	\expandafter\edef\csname bs\X\endcsname{\noexpand\boldsymbol{\X}}
\ifnum\thectr<26\repeat

\def\cite{\citet}

\makeatother

\newcommand{\g}{\gamma}

\renewcommand{\l}{\lambda}

\newcommand{\s}{\sigma}

\newcommand{\dtv}{\mathop{d_{\mathrm{TV}}}}
\newcommand{\law}{\mathop{{}\mathcal{L}}}

\newcommand{\Z}{\mathbb{Z}}

\newcommand{\Var}{\operatorname{Var}\mathopen{}}
\newcommand{\PP}{\mathop{{}\mathbb{P}}\mathopen{}}

\newcommand{\TP}{\mathop{\mathrm{TP}}\mathopen{}}
\newcommand{\EE}{\mathop{{}\mathbb{E}}\mathopen{}}
\newcommand{\I}{\mathop{{}\mathrm{I}}\mathopen{}}
\newcommand{\bigo}{\mathop{{}\mathrm{O}}\mathopen{}}
\renewcommand{\pmod}[1]{\ (\mathrm{mod}\ #1)}
\newcommand{\Po}{\mathop{\mathrm{Po}}\mathopen{}}

\begin{document}

\title [Stein's method and the modular behavior of Eulerian numbers] {Stein's method and the modular behavior of Eulerian numbers}

\author{Jason Fulman}
\address{Department of Mathematics, University of Southern California, Los Angeles, CA 90089-2532, USA}
\email{fulman@usc.edu}

\author{Adrian R\"ollin}
\address{Department of Statistics and Data Science, National University of Singapore, Singapore}
\email{adrian.roellin@nus.edu.sg}

\keywords{Stein's method, Eulerian numbers, cryptography}


\date{Version of March 21, 2026}

\begin{abstract}
The Eulerian number $A(n,k)$ counts permutations of $\{1,\dots,n\}$ with exactly $k$ descents. Motivated by problems in cryptography, several authors have studied the sum \be{\frac{1}{n!}\sum_{k \equiv j \, (\mathrm{mod}\, b)} A(n,k)} and its convergence to $1/b$. We give two proofs of explicit error bounds for this convergence, one using Stein's method for translated Poisson approximation and one using Fourier analysis. The error bound using Fourier analysis yields exponentially decaying error bounds for fixed $b$, which generalises the already known case $b=2$; however, it makes use of a special representation due to Tanny (1973). In contrast, Stein's method only yields polynomially decaying error bounds, but we hope it has potential for generalisation beyond the present setting.
\end{abstract}

\maketitle 

\section{Introduction}

We say that a permutation $\pi$ on $n$ symbols has a \emph{descent} at position $i$, where $1 \leq i \leq n-1$, if $\pi(i) > \pi(i+1)$. For example, the permutation $\underline{4} 1 \underline{5} 2 3$ has descents at positions $1$ and $3$, giving a total of $2$ descents. The \emph{Eulerian number} $A(n,k)$ is the number of permutations on $n$ symbols with exactly $k$ descents; for instance, $A(3,0)=1$, $A(3,1)=4$, $A(3,2)=1$. We refer to \citet{Pet} for a comprehensive treatment of Eulerian numbers, and to \citet{BD}, \citet{DFBook} and \citet{Ho} for their connections to card shuffling and carries in addition.

For $b \geq 2$ and $0 \leq j \leq b-1$, consider the quantity
\ben{ \label{eq:modular}
  \frac{1}{n!} \sum_{k \equiv j \pmod b} A(n,k),
}
which is the probability that the number of descents of a uniform random permutation of $\{1,\dots,n\}$ is congruent to $j$ modulo $b$. This quantity arises naturally in the cryptography literature on carries in addition; see \citet{StMe}, \citet{Sark} and \citet{Alq}. For $b=2$, the well known relation between Eulerian numbers and the Bernoulli numbers $B_n$ gives the exact formula
\be{
  \frac{1}{n!} \sum_{\text{$k$ even}} A(n,k) = \frac{1}{2} \bbbclr{ 1 + \frac{2^{n+1}(2^{n+1}-1) B_{n+1}}{(n+1)!} },
}
from which it follows by well known asymptotics of Bernoulli numbers that the convergence to $1/2$ is at rate $\bigo\bclr{(2/\pi)^n}$; see \citet{Alq}. Note that if $n \geq 2$ is even, $B_{n+1}=0$ so the error term is $0$ (which is also obvious from symmetry of the Eulerian numbers). For general fixed $b$, \citet[p.~721]{Tan} proved that~\eqref{eq:modular} converges to $1/b$ as $n \to \infty$. Tanny's argument uses the classical fact that $(1/n!) A(n,k)$ equals the probability that the sum of $n$ independent $\mathrm{Uniform}[0,1]$ random variables lies in $[k,k+1)$; the convergence then follows by wrapping this sum around a circle of circumference $b$ and applying the central limit theorem.

The goal of this note is to provide new explicit error bounds for this convergence. We provide two error bounds, one using Stein's method for translated Poisson approximation and one using Fourier analysis. Unlike for Tanny's original result, $b$ is allowed to grow with $n$. The proof via Stein's method combines two ingredients: a bound on the modular behavior of the Poisson distribution (Theorem~\ref{thm:poisson}), and a translated Poisson approximation for the number of descents via Stein's method of exchangeable pairs (Theorem~\ref{thm:descent}). The number of descents is well known to be asymptotically normal, and a proof via Stein's method was given by \citet{Fu}. However, since the number of descents is integer valued and we are interested in its residue modulo $b$, normal approximation is not directly useful; instead, we approximate the number of descents by the translated Poisson distribution in total variation, which for large variances serves as a discrete analogue of the normal distribution, and may be of independent interest. We apply \citet[Theorem~3.1]{Rol}, which provides a framework for translated Poisson approximation via exchangeable pairs, building on the exchangeable pair construction of~\citet{Fu}. We refer to \citet{Ross} for an introduction into Stein's method.

The second proof is a direct proof via Fourier analysis, but uses a very special representation of $A(n,r)$ in terms of independent uniform random variables due to \cite{Tan}. The resulting bound is exponentially small, and for $b=2$, our Corollary \ref{cor:fourier-easy} recovers Alqui\'{e}'s sharp rate $\bigo\bclr{(2/\pi)^n}$. However this argument uses independence in a crucial way whereas our Stein's method approach does not.

We conclude by noting that the cryptography papers cited above are also interested in the behavior of the $i$-th carry for fixed $i$. The discussion above concerns the $i \to \infty$ limit, in which Eulerian numbers arise. Theorem~3.4 of \citet{DiaFul} shows that if one works base $b$, then slightly more than $\frac{1}{2} \log_b(n)$ steps suffice for the total variation convergence of Holte's carries Markov chain when $n$ is large. For $i$ this large, the probability that the $i$-th carry is congruent to $j$ modulo $b$ is therefore close to~\eqref{eq:modular}.

\section{An error bound via Stein's method} 

Since we are working with permutations from the uniform distribution, it is equivalent to work with the
random variable $W(\pi)$ which is the number of descents of $\pi^{-1}$. This is equal to the number of $i$ ($1 \leq i \leq n-1)$
such that $i$ and $i+1$ are out of order in $\pi$.

\begin{theorem} \label{thm:main} Suppose that $n \geq 6$.
 Let $b \ge 2$ and $k \in \{0,1,\dots,b-1\}$. Then
\bes{
  \MoveEqLeft\bbbabs{ \frac{1}{n!} \sum_{r \equiv k \pmod b} A(n,r) - \frac{1}{b} }\\
  &\leq
  \sqrt{\frac{23}{5}} \cdot \frac{1}{\sqrt{n+1}} + \frac{24}{n+1} + \frac{b-1}{b}
  \exp\bbclr{-\frac{n+1}{12}\bclr{1-\cos(2\pi/b)}}.
}
\end{theorem}

The proof of the above result relies on two ingredients. The first is a bound on the modular behavior of the Poisson distribution, which is of independent interest. The second is an approximation of $W$ by a translated Poisson distribution, which we obtain by applying Theorem 3.1 of \citet{Rol}.

\begin{theorem} \label{thm:poisson}
Let $X \sim \Po(\l)$ and let $b \ge 2$. Then for any
$k \in \{0,1,\dots,b-1\}$,
\bes{
  \MoveEqLeft\bbbabs{ \PP\bcls{X \equiv k  \pmod b} - \frac{1}{b} }\\
  &\leq
  \frac{1}{b} \sum_{j=1}^{b-1} \exp\bclr{-\l(1-\cos(2\pi j/b))}
  \leq
  \frac{b-1}{b}
  \exp\bclr{-\l(1-\cos(2\pi/b))}.
}
\end{theorem}

\begin{proof}
We use Fourier analysis on the finite abelian group $\Z/b\Z$.
For any integer-valued random variable $X$ and any $k \in \Z$,
the indicator of the congruence class $\{X \equiv k \pmod b\}$ admits
the Fourier expansion
\be{
\I\bcls{X \equiv k \pmod b}
=
\frac{1}{b}\sum_{j=0}^{b-1}
\exp\bbbclr{\frac{2\pi i j (X-k)}{b}}.
}
Indeed, for a fixed integer $x$,
\be{
\frac{1}{b}\sum_{j=0}^{b-1}
\exp\bbbclr{\frac{2\pi i j (x-k)}{b}}
=
\begin{cases}
1, & x \equiv k \pmod b,\\
0, & x \not\equiv k \pmod b,
\end{cases}
}
by orthogonality of characters of $\Z/b\Z$.
Taking expectations yields
\be{
\PP\bcls{X \equiv k \pmod b}
=
\frac{1}{b}\sum_{j=0}^{b-1}
\EE\bclc{e^{{2\pi i j (X-k)}/{b}}}.
}
The term $j=0$ equals $1/b$. For $j \neq 0$, using the characteristic
function of the Poisson distribution,
\be{
\EE\bclc{e^{2\pi i j X/b}}
=
\exp\bclr{\l\clr{e^{2\pi i j/b}-1}}.
}
Taking absolute values gives
\be{
\babs{\EE\bclc{e^{2\pi i j (X-k)/b}}}
=
\bbabs{\exp\bclr{\l\clr{e^{2\pi i j/b}-1}}}\, \babs{e^{2\pi i j k/b}} = e^{-\l(1-\cos(2\pi j/b))}.
}
Summing over $j=1,\dots,b-1$ yields the first inequality. The second inequality follows by bounding the sum by $b-1$ times its largest term.
\end{proof}

\begin{remark}
For $b=2$, the above proof gives the exact formula
\be{
\bbbabs{ \PP\cls{X \equiv k \pmod 2} - \frac{1}{2} }
=
\frac{e^{-2\l}}{2} ,
\qquad k \in \{0,1\}.
}
\end{remark}

We proceed as \citet{Rol} and introduce the following notion. An integer-valued random variable $Y$ is said to have a \emph{translated Poisson distribution} with parameters
$\mu$ and $\s^2$, written
\be{
  \law(Y) = \TP(\mu,\s^2)
}
if
\be{
  \law(Y - \mu + \s^2 + \g) = \Po(\s^2 + \g),
}
where $\g = \langle \mu - \s^2 \rangle$ denotes the fractional part of
$\mu - \s^2$, with $\langle x \rangle = x - \lfloor x \rfloor$.
In particular, $\TP(\s^2,\s^2) = \Po(\s^2)$. Approximating a random variable $W$ by $\TP(\mu,\s^2)$ allows the mean to be matched exactly. The variance, however, can only be matched up to an additive error of at most $1$, where
\be{
\s^2 \leq \Var Y = \s^2 + \g \leq \s^2 + 1.
}
This discrepancy causes no difficulty, since for the total variation distance, the resulting error is of order $\bigo(\s^{-2})$.

The next goal is to approximate $W(\pi)$, the number of descents of the inverse of $\pi$, by a translated Poisson distribution. We apply the following result, which is a special case of Theorem 3.1 of \citet{Rol}. Recall that $(W,W')$ is called an exchangeable pair if $(W,W')$ has the same distribution as $(W',W)$. We also let \be{ \dtv(P,Q) = \frac{1}{2} \sum_{x \in \cX} \abs{P(x)-Q(x)} = \max_{A \subseteq \cX} \abs{P(A)-Q(A)} } denote the total variation distance between probability distributions $P$ and $Q$ on a finite set $\cX$.

\begin{theorem}[{See \citet[Theorem 3.1]{Rol}}] \label{thm:TP}Assume that $(W,W')$ is an exchangeable pair with values in the integers, such that
\begin{itemize}
\item $\EE W=\mu$ and $\Var W=\sigma^2 < \infty$,

\item $\EE\cls{W'-\mu \given W}  = (1-\lambda)(W-\mu)$ for some $0 < \lambda < 1$,

\item $W'-W \in \{-1,0,+1\}$.
\end{itemize}
Then, with $S = S(W) = \PP\cls{W'=W+1 \given W}$,
\be{
 \dtv\bclr{\law(W),\TP(\mu,\sigma^2)} \leq \frac{\sqrt{\Var S}}{\lambda \sigma^2} + \frac{2}{\sigma^2}.
}
\end{theorem}

As a consequence of Theorem \ref{thm:TP}, we obtain the following result.

\begin{theorem} \label{thm:descent}
Let $W(\pi)$ be the number of descents of $\pi^{-1}$ where $\pi$ is chosen uniformly at random from the set of
permutations of $\{1,2,\cdots,n\}$. Then for $n \geq 6$, $\mu=\frac{n-1}{2}$ and $\sigma^2=\frac{n+1}{12}$,
\be{
\dtv\bclr{\law(W),\TP(\mu,\sigma^2)} \leq \sqrt{\frac{23}{5}} \cdot \frac{1}{\sqrt{n+1}} + \frac{24}{n+1}.
}
\end{theorem}

\begin{proof} We use an exchangeable pair $(W,W')$ constructed in \citet{Fu}. Letting $W(\pi)$ be the number of descents of $\pi^{-1}$, we define $W'=W'(\pi)$ as follows. Pick an integer $I$ uniformly at random between $1$ and $n$ and move the number in position $I$ in the second row of $\pi$ to the end of this second row. This results in a new permutation $\pi'$, and we define $W'(\pi)=W(\pi')$. For example, suppose that $n=7$ and $I=3$. Then the permutation $\pi$ in two-line form is

\ba{
i & = \ 1 \ 2 \ 3 \ 4 \ 5 \ 6 \ 7 \\
\pi(i) & = \ 6 \ 4 \ 1 \  5 \ 3 \ 2 \ 7
}
is transformed to
\ba{
i & = \ 1 \ 2 \ 3 \ 4 \ 5 \ 6 \ 7 \\
\pi'(i) & = \ 6 \  4 \ 5 \ 3 \ 2 \ 7  \ 1.
}
\citet{Fu} proved that despite the non-reversibility of the ``move random to end'' Markov chain, the pair $(W,W')$ is exchangeable. The exchangeability also follows from Lemma 1.1 of \citet{RR} and the fact that $W'-W \in \{-1,0,+1\}$.

It is well known that $\EE W=(n-1)/2$ and $\Var W=(n+1)/12$. Clearly $W'-W \in \{-1,0,+1\}$. Moreover, Lemma 1 of \citet{Fu}
yields \be{ \EE\cls{W'-\mu \given W}  = (1-\lambda)(W-\mu) } for $\lambda=2/n$. Thus the hypotheses of Theorem \ref{thm:TP} are satisfied.

To apply Theorem \ref{thm:TP} with $\lambda=2/n$ and $\sigma^2=(n+1)/12$, it remains only to bound $\Var S$. Since $W$ is a function of $\pi$, it follows from the standard conditional variance decomposition of $\PP\cls{W'=W+1 \given \pi}$ and the projection property that
\be{
  \Var S \leq \Var\bclr{\PP\cls{W'=W+1 \given \pi}}.
}
Clearly,
\be{
  \PP\cls{W'=W+1 \given \pi} = \frac{1}{n} \bclr{ X_1(\pi) + X_2(\pi) + \cdots + X_{n-1}(\pi)},
}
where
\be{X_1(\pi) =
\begin{cases}
 1 &\text{if $1$ and $2$ are in order in $\pi$,}\\
 0 &\text{else,}
\end{cases}
}
and for $2 \leq i \leq n-1$,
\be{
X_i(\pi) = \begin{cases}
1 &\text{if $i-1,i,i+1$ are in order in $\pi$,}\\
0 &\text{else.}
\end{cases}
}
Thus,
\be{ \EE\bclc{\PP\cls{W'=W+1 \given \pi}} = \frac{1}{n} \clr{\EE X_1+ \cdots + \EE X_{n-1}} = \frac{1}{n} \bbbclr{ \frac{1}{2} + \frac{n-2}{6} } = \frac{n+1}{6n}.}
Now for $n \geq 6$, one has that
\be{
  \EE\bclc{\PP\cls{W'=W+1 \given \pi}^2} = \frac{1}{n^2} \clr{T_1 + T_2 + T_3 + T_4 + T_5},
} where
\ba{
T_1 & = \EE\clc{X_1^2} + \cdots + \EE\clc{X_{n-1}^2} = \EE X_1 + \cdots + \EE X_{n-1} = \frac{n+1}{6},\\
\bs{
T_2 & = 2\EE\clc{X_1X_2} + 2\EE\clc{X_1X_3} + 2\EE\clc{X_1X_4} + \cdots + 2\EE\clc{X_1X_{n-1}} \\
& = \frac{2}{6} + \frac{2}{24} + \frac{2(n-4)}{12},
}\\
T_3 &= 2 \sum_{i=2}^{n-2} \EE\clc{X_i X_{i+1}} = \frac{2(n-3)}{24},
\qquad T_4 = 2 \sum_{i=2}^{n-3} \EE\clc{X_i X_{i+2}} = \frac{2(n-4)}{120},\\
T_5 &= 2 \sum_{i=2}^{n-4} \sum_{j=i+3}^{n-1} \EE\clc{X_iX_j} = \frac{2}{36} \sum_{i=2}^{n-4} \sum_{j=i+3}^{n-1} 1
= \frac{2}{36} \frac{(n-4)(n-5)}{2}.
}
Thus, straightforward algebraic manipulation yields
\be{
  \Var\bclr{\PP\cls{W'=W+1 \given \pi}} =  \frac{1}{n^2} \clr{T_1 + T_2 + T_3 + T_4 + T_5}  - \bbbclr{\frac{n+1}{6n}}^2 = \frac{23(n+1)}{180n^2}.
}
Plugging into Theorem \ref{thm:TP} completes the proof.
\end{proof}

We can now give the proof of our main Stein's method result.

\begin{proof}[Proof of Theorem~\ref{thm:main}]
Since $W$ is the number of descents of the inverse of a uniform random permutation, $\PP\bcls{W = r} = A(n,r)/n!$ for each $r$, and hence $\PP\bcls{W \equiv k \pmod b} = \frac{1}{n!}\sum_{r \equiv k \pmod b} A(n,r)$.
Let $Y\sim\TP(\mu,\sigma^2)$, where $\mu = (n-1)/2$ and $\sigma^2 = (n+1)/12$. By the triangle inequality,
\be{
  \bbbabs{\PP\bcls{W \equiv k \pmod b} - \frac{1}{b}}
  \leq \dtv\bclr{\law(W), \TP(\mu,\sigma^2)} + \bbbabs{\PP\bcls{Y \equiv k \pmod b} - \frac{1}{b}}.
}
The first term is bounded by $\sqrt{23/5} \cdot (n+1)^{-1/2} + 24/(n+1)$ by Theorem~\ref{thm:descent}. For the second term, note that $Y$ shifted by $\mu - \sigma^2 - \gamma$ has distribution $\Po(\sigma^2 + \gamma)$, where $\g = \langle \mu - \s^2 \rangle\in[0,1)$ and $\sigma^2 \leq \sigma^2 + \gamma < \sigma^2 + 1$, so applying Theorem~\ref{thm:poisson} with $\lambda = \sigma^2 + \gamma \geq (n+1)/12$ yields a bound of
\be{
  \frac{b-1}{b} \exp\bbclr{-\frac{n+1}{12}\bclr{1 - \cos(2\pi/b)}}.
}
Combining the two bounds completes the proof.
\end{proof}

\section{An error bound via Fourier analysis}\label{sec:fourier}

A direct Fourier-analytic argument gives the following sharper bound. It makes, however, use of a representation of $A(n,r)$ given by \citet{Tan}, and is therefore not robust to generalisation and does not give total variation bounds for the distribution of~$W$.

\begin{theorem}\label{thm:fourier}
For $n\geq 1$, $b \geq 2$, and $k \in \{0,1,\dots,b-1\}$,
\bes{
\bbbabs{\frac{1}{n!}\sum_{r\equiv k\pmod b} A(n,r) - \frac{1}{b}}
\leq \frac{1}{b}\sum_{j\neq 0}\bbabs{\frac{b\sin(\pi j/b)}{\pi j}}^{n+1}.
}
\end{theorem}

\begin{proof}
Using the representation appearing in \citet{Tan}, we have $A(n,r)/n! = \PP\bcls{\lfloor S_n\rfloor = r}$, where $S_n = U_1 + \cdots + U_n$ and $U_1,\dots,U_n$ are i.i.d.\ Uniform$[0,1]$. Hence, for $k \in \{0,1,\dots,b-1\}$,
\be{
\frac{1}{n!}\sum_{r\equiv k\pmod b} A(n,r) = \PP\bcls{\lfloor S_n\rfloor \equiv k \pmod b} = \EE{J(S_n)},
}
where $J \colon \mathbb{R}\to\{0,1\}$ is defined by $J(x) = \I\bcls{x \bmod b \in [k,k+1)}$. The function $J$ admits the Fourier expansion
\be{
J(v) = \sum_{j\in\Z} c_j\, e^{2\pi i jv/b}
}
with Fourier coefficients
\be{
c_j = \frac{1}{b}\int_0^b J(v)\, e^{-2\pi ijv/b}\,dv = \frac{1}{b}\int_k^{k+1} e^{-2\pi ijv/b}\,dv.
}
A direct computation gives $c_0 = 1/b$ and
\be{
|c_j| = \frac{|\sin(\pi j/b)|}{\pi |j|}, \qquad j\neq 0.
}
Since $U \sim \mathrm{Uniform}[0,1]$ has characteristic function $\phi_U(t) = (e^{it}-1)/(it)$, the independence of $U_1,\dots,U_n$ gives
\be{
\babs{\EE{e^{2\pi i j S_n/b}}} = |\phi_U(2\pi j/b)|^n = \bbabs{\frac{b\sin(\pi j/b)}{\pi j}}^n.
}
The product $|c_j|\cdot|\phi_{S_n}(2\pi j/b)| = \frac{1}{b}|b\sin(\pi j/b)/(\pi j)|^{n+1}$ is of order $\bigo(|j|^{-(n+1)})$ and therefore summable. By Fej\'er's theorem (the Ces\`aro means of the Fourier series of $J$ are bounded by $\|J\|_\infty = 1$ and converge to $J$ almost everywhere) together with bounded convergence,
\be{
\EE{J(S_n)} = \sum_{j\in\Z} c_j\, \EE {e^{2\pi i j S_n/b}}.
}
The $j = 0$ contribution equals $1/b$, and the triangle inequality yields
\be{
\bbbabs{\EE{J(S_n)} - \frac{1}{b}}
\leq \sum_{j\neq 0}|c_j|\cdot |\phi_{S_n}(2\pi j/b)|= \frac{1}{b}\sum_{j\neq 0}\bbabs{\frac{b\sin(\pi j/b)}{\pi j}}^{n+1},
}
which is the claimed bound.
\end{proof}

\begin{corollary}\label{cor:fourier-easy}
For $n\geq 1$, $b\geq 2$, and $k \in \{0,1,\dots,b-1\}$,
\bes{
\bbbabs{\frac{1}{n!}\sum_{r\equiv k\pmod b}A(n,r) - \frac{1}{b}} \leq \frac{b-1}{b}\,\rho(b)^{n-1},
}
where $\rho(b) := b\sin(\pi/b)/\pi \in (0,1)$. Moreover,
\be{
\rho(b)^{n-1} \leq \exp\bbclr{-\frac{(n-1)\pi^2}{8b^2}}.
}
\end{corollary}

\begin{proof}
The function $\sin(x)/x$ is decreasing on $(0,\pi]$, so $|b\sin(\pi j/b)/(\pi j)|\leq \rho(b)$ for $1\leq |j|\leq b$. For $|j|>b$, we have $|b\sin(\pi j/b)/(\pi j)|\leq b/(\pi|j|) < 1/\pi \leq \rho(b)$. Hence, for $n\geq 1$,
\be{
\bbabs{\frac{b\sin(\pi j/b)}{\pi j}}^{n+1} \leq \rho(b)^{n-1}\,\bbabs{\frac{b\sin(\pi j/b)}{\pi j}}^2.
}
Parseval's identity applied to $J$ gives $\sum_{j\in\Z}|c_j|^2 = b^{-1}\int_0^b |J(v)|^2\,dv = 1/b$, so $\sum_{j\neq 0}|c_j|^2 = (b-1)/b^2$, and since $|b\sin(\pi j/b)/(\pi j)| = b|c_j|$ for $j\neq 0$, we have
\be{
\sum_{j\neq 0}\bbabs{\frac{b\sin(\pi j/b)}{\pi j}}^2 = b^2\sum_{j\neq 0}|c_j|^2 = b-1.
}
Combining these inequalities with Theorem~\ref{thm:fourier} yields the first inequality.

For the second inequality, set $u = \pi/b \in (0,\pi/2]$. The alternating Taylor expansion of $\sin$ gives $\sin u \leq u - u^3/6 + u^5/120$, and on $(0,\pi/2]$ this is in turn bounded by $u - u^3/8$, as is easily verified. Hence $1 - \rho(b) = 1 - \sin(u)/u \geq u^2/8 = \pi^2/(8b^2)$, and the bound $1 + x \leq e^x$ with $x = \rho(b) - 1$ yields $\rho(b)^{n-1} \leq e^{-(n-1)(1-\rho(b))} \leq \exp\bclr{-(n-1)\pi^2/(8b^2)}$.
\end{proof}

\begin{remark}
For $b=2$ Corollary \ref{cor:fourier-easy} recovers Alqui\'{e}'s sharp $\bigo\bclr{(2/\pi)^n}$ result.
\end{remark}

\begin{remark}
In the borderline regime $b \asymp \sqrt{n}$, Corollary~\ref{cor:fourier-easy} no longer yields convergence to $0$. Indeed, if $b = c\sqrt{n}$, then $\rho(b)^{n-1}$ tends to the positive constant $\exp\bclr{-\pi^2/(6c^2)}$. In this case, we can apply Theorem~\ref{thm:fourier} directly: Each term $|b\sin(\pi j/b)/(\pi j)|^{n+1}$ converges to $\exp\bclr{-\pi^2 j^2/(6c^2)}$, so the sum is $\bigo(1)$ and the resulting bound is $\bigo(1/b) = \bigo(1/\sqrt{n})$.
\end{remark}

\section*{Acknowledgements}
Fulman was supported by Simons Foundation grant 917224.

\end{document}